\documentclass[oneside,english]{amsart}
\usepackage[english]{babel}
\usepackage{amsthm,gensymb,hyperref}
\usepackage{tikz}
\usepackage{tikz}
\usetikzlibrary{matrix}
\usetikzlibrary{positioning,arrows.meta} 
\usetikzlibrary{positioning}
\usepackage{multicol}
\usepackage{extarrows}%xlongequal
\usepackage{mathtools}%xrightarrow
\usepackage{commath}%norm
\usepackage{amssymb}
\usepackage{amscd}
\newtheorem{theorem}{Theorem}[section]
\theoremstyle{definition}
\newtheorem{definition}[theorem]{Definition}
\theoremstyle{plain}

\theoremstyle{plain}

\theoremstyle{plain}

\theoremstyle{remark}

\theoremstyle{remark}

\theoremstyle{remark}
\begin{document}

\title[Twelve characterisations of Killing vector fields in Finsler geometry]{Twelve characterisations of Killing vector fields in Finsler geometry}

\author[Bucataru]{Ioan Bucataru}
\address{Faculty of  Mathematics \\ Alexandru Ioan Cuza University \\ Ia\c si, 
  Romania}
\email{bucataru@uaic.ro}
\urladdr{http://orcid.org/0000-0002-8506-7567}

\author[Cre\c tu]{Georgeta Cre\c tu}
\address{Department of  Mathematics \\ Gheorghe Asachi Technical University \\ Ia\c si, 
  Romania}
\email{cretuggeorgeta@gmail.com}
\urladdr{http://orcid.org/0000-0003-4197-0268}

\author[Daniliuc]{Luciana Daniliuc}
\address{Faculty of  Mathematics \\ Alexandru Ioan Cuza University \\ Ia\c si, 
  Romania}
\email{daniliucluciana@yahoo.com.au}
%\urladdr{http://orcid.org/0000-0002-8506-7567}

\begin{abstract}

It is well known that infinitesimal isometries preserve many of the canonical geometric structures of a Riemannian or Finslerian manifold. In this paper, we identify those geometric structures of a Finslerian manifold that are preserved by an infinitesimal isometry (a Killing vector field) and whose preservation characterises such fields. Using the Fr\"olicher–Nijenhuis formalism as a unified geometric framework, we establish twelve equivalent characterisations of the Killing property in the Finslerian setting, utilizing the Finsler norm, the sphere bundle, contact and symplectic structures, Reeb and Hamiltonian vector fields, first integrals, the homogenized almost complex structure, and the Sasaki metrics. Furthermore, we derive the local expressions of these characterisations, yielding several equivalent forms of the Killing equations, including scalar, one-form and tensorial formulations. In our approach, none of these formulations is regarded as primary; each reflects a different geometric aspect of the Killing condition.
\end{abstract}


  \maketitle
\section{Introduction}

The study of infinitesimal isometries traces back to the work of W. Killing, who first derived the conditions under which the flow of a vector field preserves a Riemannian metric, leading to what are now known as Killing vector fields, characterized by the Killing equations \cite{K1890}. 

Since then, Killing vector fields have been studied intensively. An important contribution to this topic is due to Sasaki \cite{S1958}, who studied the geometry of the tangent bundle of a Riemannian manifold and characterised Killing vector fields in terms of their lifts, which are Killing vector fields with respect to what is now called the Sasaki metric. These aspects were recently revisited and extended by Albuquerque in \cite{A2019}, who showed that Killing vector fields preserve several other geometric structures living on the tangent bundle: contact and symplectic structures.

The geometric structures associated with a Finsler manifold are defined on the slit tangent bundle $T_0M$ rather than on the manifold itself. Consequently, the notion of infinitesimal isometries must also be formulated on the tangent bundle. The first extension of these concepts to Finsler geometry is due to Yano in \cite[Chaper VIII]{Y1955}, where he defines a motion (isometry) in a Finsler space as an extended point transformation which does not change the Finsler function. Consequently, Killing vector fields are naturally described by their complete lifts, whose flows act by isometries   \cite{AK2014}, \cite{C2023}, \cite{DH2002}, \cite{L2004}. In a recent paper \cite{C2023}, Crampin showed that in the Finslerian context, the complete lift of a Killing vector field preserves many induced geometric structures: the geodesic spray, the Sasaki metric, the homogenized almost-complex structure, and hence the almost K\"ahlerian structure.

The main purpose of this paper is to identify those canonical geometric structures whose invariance under the complete lift of a vector field is equivalent to the Killing property. In Theorem \ref{thmKglobal}, we propose twelve such characterisations, which state that a vector field $X\in \mathfrak{X}(M)$ is a Killing vector field if and only if its complete lift either preserves the Finsler norm, the sphere bundle, the contact and symplectic structures, the Sasaki metric, the Reeb vector field, the homogenized almost-complex structure, or acts as a Hamiltonian vector field that induces a first integral. 

Our global approach is based entirely on the Frölicher–Nijenhuis formalism \cite{GM2000,SLK2014}, which provides a unified geometric framework for describing the canonical structures of a Finsler manifold. This allows us to derive the twelve equivalent characterisations from a common geometric setting rather than treating them separately.

In Theorem \ref{thmKlocal}, we reformulate the first nine conditions of Theorem \ref{thmKglobal} in local coordinates to provide various expressions of Killing equations. We obtain three types of Killing equations: four scalar Killing equations, three $1$-form Killing equations and two tensorial Killing equations.

\section{Fr\"olicher-Nijenhuis formalism on Finsler manifolds}

The twelve characterisations and the proof of Theorem \ref{thmKglobal} rely entirely on the Fr\"olicher--Nijenhuis formalism, including the Fr\"olicker-Nijenhuis bracket, algebraic and Lie derivations, and their commutation formulae. 

\subsection{A short introduction to the Fr\"olicher-Nijenhuis formalism.} For an $n$-dimensional manifold $M$, we denote by $C^{\infty}(M)$ the ring of smooth functions, and by $\mathfrak{X}(M)$, $\Lambda(M)$, and $\Psi(M)$ the Lie algebra of vector fields,  the Lie algebra of forms, and  the graded Lie algebra of vector-valued forms, respectively. 

We will use the Fr\"olicher-Nijenhuis formalism (FN for short) to derive all the geometric structures associated to a Finsler manifold, \cite{GM2000, SLK2014}. We begin by briefly recalling the basic concepts of the FN formalism. 

\begin{definition}
A derivation of degree $r\in \mathbb{Z}$ on $\Lambda(M)$ is an $\mathbb{R}$-linear map $D: \Lambda^k(M) \to \Lambda^{k+r}(M)$, which is a local operator and satisfies the graded Leibniz rule:

$D(\omega \wedge \theta) =D\omega \wedge \theta + (-1)^{kr}\omega \wedge D\theta, \forall \omega \in \Lambda^k(M).$
\end{definition}
If we consider the commutator of two derivations $D_1$ and $D_2$ of degrees $r_1$ and $r_2$,  
\begin{eqnarray*}
[D_1, D_2]=D_1\circ D_2 - (-1)^{r_1r_2} D_2\circ D_1,
\end{eqnarray*} 
then it satisfies the graded Jacobi identity:
\begin{eqnarray*} 
(-1)^{r_1r_3}[D_1,[D_2,D_3]] + (-1)^{r_2r_1}[D_2,[D_3,D_1]] + (-1)^{r_3r_2}[D_3,[D_1,D_2]]=0.
\end{eqnarray*} 
Consequently, $\operatorname{Der}(\Lambda(M))$ is a graded Lie algebra.

Due to their properties, any derivation is uniquely determined by its action on functions and $1$-forms. The Fr\"olicher-Nijenhuis formalism relies on the fact that there are only two types of derivations, an algebraic derivation and a Lie derivation, and each of them can be identified with a vector-valued form \cite{FN1956, KMS1993}.

Algebraic derivations (or $i_*$-derivations) are trivial on $C^{\infty}(M)$ and hence are uniquely determined by their action on $1$-forms. For a vector-valued $l$-form $L\in \Psi^l(M)$, the  corresponding unique  algebraic derivation is of degree $l-1$ and is given by $i_L: \Lambda^k(M) \to \Lambda^{k+l-1}(M)$, 
\begin{eqnarray*}
 & i_L\omega(X_1,...,X_{l+k-1}) = \\
 & \frac{1}{l!(k-1)!} 
\sum_{\sigma \in \mathcal{S}_{l+k-1}} \varepsilon(\sigma) \omega(L(X_{\sigma(1)},..., X_{\sigma(l)}), X_{\sigma(l+1)},...,X_{\sigma{(l+k-1)}}).
\end{eqnarray*}

Lie derivations (or $d_*$-derivations) commute with the exterior derivative $d$ and hence are uniquely determined by their action on functions. For a vector-valued $l$-form $L\in \Psi^l(M)$, the  corresponding unique Lie derivation is of degree $l$ and is given by $d_L: \Lambda^k(M) \to \Lambda^{k+l}(M)$, 
\begin{eqnarray*}
d_L=[i_L, d].
\end{eqnarray*}

For two vector-valued forms $L\in \Psi^l(M)$ and $K\in \Psi^k(M)$, the commutator of their induced Lie derivations $[d_L, d_K]$ is the Lie derivation of a uniquely defined $(k+l)$-form $[L,K]\in \Psi^{k+l}(M)$, which is called the Fr\"olicher-Nijenhuis bracket of $L$ and $K$:
\begin{eqnarray}
[d_L,d_K]=d_{[L,K]}. \label{dlk}
\end{eqnarray}
The FN-bracket endows $\Psi(M)$ with the structure of a graded Lie algebra, isomorphic to the graded Lie algebra of Lie derivations on  $\Lambda(M)$. 

For a vector-valued form $L\in \Psi^l(M)$, its Nijenhuis torsion is given by 
\begin{eqnarray*}
N_L=\frac{1}{2}[L,L],
\end{eqnarray*} 
and characterises its integrability. For $X\in \mathfrak{X}(M)$ and $L\in \Psi(M)$, the FN-bracket $[X, L]=\mathcal{L}_XL$ is the usual Lie derivative.

\subsection{The geometric setting of a Finsler manifold.} For an $n$-dimensional manifold $M$, we consider its tangent bundle $TM$, which is a $2n$-dimensional manifold, endowed with local coordinates $(x^i, y^i)$. We will also consider $T_0M=TM\setminus \{0\}$, the tangent bundle with the zero section removed. 

On $TM$ there are two important natural structures. The tangent structure (or vertical endomorphism) $\mathcal{T}=\frac{\partial}{\partial y^i}\otimes dx^i$ is an integrable vector-valued $1$-form, meaning that its Nijenhuis torsion vanishes: $[\mathcal{T},\mathcal{T}]=0$, and it has the property $\mathcal{T}^2=0$. Throughout this work, we will use the summation convention over repeated covariant and contravariant indices. 

The Liouville (dilation) vector field $\mathcal{C}= y^i \frac{\partial}{\partial y^i}$ plays a fundamental role in characterising, via Euler's theorem, positively homogeneous tensors on $T_0M$ with respect to the fibre coordinates. A tensor $A$ on $T_0M$ is positively homogeneous of order $p\in \mathbb{Z}$ if $\mathcal{L}_{\mathcal{C}}A=pA$. 

One can easily check that $[\mathcal{C}, \mathcal{T}]=-\mathcal{T}$, and hence the tangent structure $\mathcal{T}$ is $(-1)$-homogeneous. 

\begin{definition} \label{defnFinsler}
A Finsler function $F:TM \to \mathbb{R}$ is given by a family of norms $F: x\in M \to F_x: T_xM \to \mathbb{R_+}$ that satisfies:
\begin{itemize}
\item[$F_1)$] $F$ is smooth on $T_0M=TM \setminus \{0\}$ and continuous on $TM$;
\item[$F_2)$] $F$ is positively $1$-homogeneous: $F(x,\lambda y)=\lambda F(x,y)$, $\forall \lambda \geq 0$;
\item[$F_3)$] the metric tensor $g_{ij}(x,y)=\dfrac{1}{2}\dfrac{\partial^2 F^2}{\partial y^i \partial y^j}(x,y)$ is non-degenerate (or positive definite).
\end{itemize}
\end{definition}
 
The homogeneity of the Finsler metric yields several identities that relate the Finsler norm $F$ and the metric tensor $g_{ij}$, which are very similar to those of the Riemannian (quadratic) case:
\begin{eqnarray}
F^2(x,y)=g_{ij}(x,y)y^iy^j, \quad y_i:=g_{ij}y^j=F\frac{\partial F}{\partial y^i}. \label{fgij}
\end{eqnarray} 
Throughout this paper, indices are raised and lowered using the metric tensor $g_{ij}$. For example, if $X^i$ are the components of a vector field, we denote by $X_i:=g_{ij}X^j$ the components of the corresponding covector field. 

Alongside the metric tensor $g_{ij}$, we consider the \emph{angular metric tensor}:
\begin{eqnarray}
h_{ij}=F\frac{\partial^2F}{\partial y^i\partial y^j}=g_{ij}-\frac{\partial F}{\partial y^i} \frac{\partial F}{\partial y^j} = g_{ij}-\frac{1}{F^2} y_iy_j. \label{hij} 
\end{eqnarray}
By homogeneity of the Finsler function, we obtain that $h_{ij}y^i=0$ and the regularity condition $F_3)$ is equivalent to the fact that the angular metric has rank $n-1$ (or is positive semi-definite), \cite{C2011}.  
 
The regularity condition $F_3)$ is equivalent to the fact that the Euler-Lagrange $2$-form of the energy function $E=\frac{1}{2}F^2$, 
 \begin{eqnarray}
\omega_E=dd_\mathcal{T}E=d(g_{ij}y^j)\wedge dx^i \label{omegae}
\end{eqnarray}
is a symplectic $2$-form on $T_0M$. 

The variational problem for the energy function $E$ leads to the Euler-Lagrange equations, which uniquely determine a second-order vector field (the geodesic spray) $S\in \mathfrak{X}(T_0M)$:
\begin{eqnarray} \label{lsdje}
\mathcal{L}_Sd_\mathcal{T}E-dE=0.
\end{eqnarray}   
The Euler-Lagrange equation \eqref{lsdje} can be reformulated using the symplectic $2$-form as
\begin{eqnarray} \label{isddje}
i_Sdd_\mathcal{T}E=-dE,
\end{eqnarray}
showing that the geodesic spray $S$ is the Hamiltonian vector field associated with the energy function $E$. Directly, from the equation \eqref{isddje}, we find that $S(E)=0$, meaning that the energy is constant along the geodesics.

The geodesic spray $S$ is $1$-homogeneous, $[\mathcal{C}, S]=S$, and has the local expression: 
\begin{eqnarray*}
S=y^i\frac{\partial}{\partial x^i}  - 2G^i(x,y)\frac{\partial}{\partial y^i}.
\end{eqnarray*}
The homogeneity property of the spray $S$ is equivalent to the $2$-homogeneity of the spray coefficients $G^i(x,y)$.

The FN-bracket $\Gamma=[S,\mathcal{T}]$ was introduced by Grifone in \cite{G1972} as a nonlinear connection on $T_0M$.  It satisfies $\Gamma^2=\operatorname{Id}$ and therefore is an almost product structure with two eigenvalues $1$ and $-1$. The eigendistribution corresponding to $1$ is the vertical distribution (tangent to the fibres of the tangent bundle), while the eigendistribution corresponding to $-1$ will be called the horizontal distribution.  Therefore, the almost product structure can be expressed as follows:     
\begin{eqnarray}
\left[S, \mathcal{T}\right] & = & -h+v= -\frac{\delta}{\delta x^i} \otimes dx^i + 
 \frac{\partial}{\partial y^i} \otimes \delta y^i, \quad \textrm{where} \label{gamma}  \\
\frac{\delta}{\delta x^i}  & = & \frac{\partial}{\partial x^i} - N^j_i(x,y) \frac{\partial}{\partial y^j}, \quad \delta y^i = dy^i + N^i_j(x,y)dx^j, \quad N^i_j=\frac{\partial G^i}{\partial y^j}. \nonumber
\end{eqnarray} 
The almost product structure $\Gamma$, along with the horizontal and vertical projectors $h$ and $v$, are $0$-homogeneous, which is equivalent to the fact that the nonlinear coefficients $N^i_j(x,y)$ are $1$-homogeneous.

In formulae \eqref{gamma}, $\delta/\delta x^i$ are local generators for the horizontal distribution, while $\delta y^i$ are local annihilators for the horizontal distribution. We will refer to the basis $\{\delta/\delta x^i, \partial/\partial y^i\}$ as a basis adapted to the horizontal and vertical distributions. Its dual basis is $\{dx^i, \delta y^i\}$, which is very useful for expressing various geometric structures. For example, the symplectic structure \eqref{omegae} has the following expression:
\begin{eqnarray}
\omega_E=g_{ij} \delta y^i \wedge dx^j. \label{oedelta}
\end{eqnarray}
Expression \eqref{oedelta} shows that both the horizontal and vertical distributions are Lagrangian distributions with respect to $\omega_E$.

The geodesic spray $S$ is a horizontal vector field, $hS=S$, and this can be justified as follows:
\begin{eqnarray*}
-hS+vS=[S, \mathcal{T}](S)=[S, \mathcal{T}(S)]-\mathcal{T}[S,S]=[S, \mathcal{C}]=-S.
\end{eqnarray*}
The fact that the geodesic spray is horizontal can also be seen from its local expression: $S=y^i\delta/\delta x^i$.

An important result in Finsler geometry is the reformulation of the Euler-Lagrange equation \eqref{lsdje} (second-order PDE) in terms of the first-order PDE $d_hF=0$, which is equivalent to $d_hE=0$.
One can easily verify this equivalence as follows. The commutation formula \eqref{dlk} for a spray $S$ and the tangent structure $\mathcal{T}$ reads: $[\mathcal{L}_S, d_{\mathcal{T}}]=d_{[S, \mathcal{T}]}=d_{\Gamma}= d_{-h+v}.$ Therefore, the Euler-Lagrange equations \eqref{lsdje} are equivalent to the following equations:
\begin{eqnarray}
d_{\mathcal{T}}\mathcal{L}_SE + d_{\Gamma}E-dE=0 \stackrel{S(E)=0}{\Longleftrightarrow} d_hE=0 \Longleftrightarrow d_hF=0. \label{dhe}
\end{eqnarray} 

Another important structure induced by a Finsler manifold and its geodesic spray $S$ is the \emph{almost complex structure} on $T_0M$, given by: 
\begin{align}
J=h\circ [S, h] - \mathcal{T} = \frac{\delta}{\delta x^i} \otimes
  \delta y^i - \frac{\partial}{\partial y^i} \otimes dx^i. \label{acs}
\end{align}
As observed by Crampin in \cite{C2008}, the almost complex structure does not have a homogeneity degree. This can be rectified by considering
\begin{align}
J_0=\frac{1}{F} h\circ [S, h] - F\mathcal{T} = \frac{1}{F}\frac{\delta}{\delta x^i} \otimes
  \delta y^i - F \frac{\partial}{\partial y^i} \otimes dx^i. \label{acs0}
\end{align} 
The structure $J_0$ defined this way is $0$-homogeneous, meaning that $\mathcal{L}_{\mathcal C}J_0=0$. We will see in the next section that the homogenized almost complex structure $J_0$ is more useful for characterising Killing vector fields.

Using the almost complex structure $J$, the symplectic $2$-form $\omega_E$ determines a (semi-)Riemannian metric on $T_0M$, given by
\begin{align} & G(X, Y) = \omega_{E}(X, JY), \quad
  \omega_{E}(X,Y)= G(\mathcal{T}X,Y) - G(X,\mathcal{T}Y), \label{gomega} \\ 
&  G = g_{ij}dx^i\otimes dx^j + g_{ij}\delta y^i\otimes \delta y^j. \nonumber
\end{align}
The metric $G$ is called the Sasaki metric induced by the Finsler function $F$. With respect to this metric, the horizontal and vertical distributions are mutually orthogonal. 

The Sasaki metric $G$ is not homogeneous. We can correct this by replacing the almost complex structure $J$ with the homogenized almost complex structure $J_0$ in formulae \eqref{gomega}:  
\begin{align} & G_1(X, Y) = \omega_{E}(X, J_0Y), \quad
  \omega_{E}(X,Y)= G_1(\mathcal{T}X,Y) - G_1(X,\mathcal{T}Y), \label{g1omega} \\ 
&  G_1 = Fg_{ij}dx^i\otimes dx^j + \frac{1}{F}g_{ij}\delta y^i\otimes \delta y^j. \nonumber
\end{align}
The Sasaki metric $G_1$ has been considered by Crampin in \cite{C2008} and it is $1$-homogeneous.

The non-degeneracy condition $F_3)$ of Definition \ref{defnFinsler} is also equivalent to the fact that the Hilbert $1$-form 
\begin{eqnarray*}
d_\mathcal{T}F=\frac{\partial F}{\partial y^i}dx^i=\frac{1}{F}g_{ij}y^jdx^i
\end{eqnarray*}
 is a contact $1$-form on the $(2n-1)$-dimensional sphere bundle $SM=\{(x,y)\in T_0M \mid F(x,y)=1\}$. Since the the restriction of the normalized geodesic spray $S/F$ to $SM$ satisfies
\begin{eqnarray*}
i_{S/F}d_{\mathcal{T}}F=1, \quad i_{S/F}dd_{\mathcal{T}}F=0,
\end{eqnarray*}
it follows that $S/F$ is the Reeb vector field of the contact manifold $(SM, d_{\mathcal{T}}F)$.

For the hypersurface $SM\subset T_0M$, its tangent bundle $TSM$ gives a $(2n-1)$-distribution on $T_0M$, which is orthogonal to the Liouville vector field $\mathcal{C}$, with respect to either Sasaki metric $G$ or $G_1$.

\section{Killing vector fields}

\subsection{Complete lifts and geometric formulations}
In Riemannian geometry, a vector field is a Killing vector field if its flow consists of isometries \cite{K1890}. In other words, a vector field is Killing if and only if its corresponding Lie transport preserves the Riemannian metric $g$, or the induced norm, along its integral curves. This leads to the classical characterisation of Killing vector fields for a Riemannian metric $g$: $$\mathcal{L}_Xg=0.$$

In Finsler geometry, most of the induced geometric structures are defined on the slit tangent bundle $T_0M$, or on the sphere bundle $SM$. Therefore, when extending the notion of Killing vector fields to the Finslerian setting, we must look for isometries acting on $T_0M$. 

For a vector field $X\in \mathfrak{X}(M)$ with local flow $\varphi_t=\operatorname{Fl}_t^X$, the flow of its complete lift $X^c\in \mathfrak{X}(TM)$ is given by $(\varphi_t, (\varphi_t)_{\ast})=\operatorname{Fl}_t^{X^c}$. According to Yano \cite{Y1955}, we consider  the following definition of a Finslerian Killing vector field.
\begin{definition}
A vector field $X \in \mathfrak{X}(M)$ is called a \emph{Killing vector field} for a Finsler manifold $(M, F)$ if the flow of its complete lift $X^c$ preserves the Finslerian norm:
\begin{equation*}\label{eq:DefKilling}
F\left(\varphi_t(x), (\varphi_t)_{*,x}(y)\right) = F(x,y), \quad \forall (x,y) \in T_0M, \; \forall t.
\end{equation*}
\end{definition}
This yields the classical characterisation in the Finslerian setting: a vector field $X\in \mathfrak{X}(M)$ is a Killing vector field for a Finsler metric $F$ if and only if $\mathcal{L}_{X^c}F=0$. 

For a vector field $X\in \mathfrak{X}(M)$, its complete lift $X^c$ is the unique vector field on $TM$ characterised by the following properties:
\begin{itemize}
\item[$CL_1)$] It is $0$-homogeneous: $[\mathcal{C}, X^c]=0$.
\item[$CL_2)$] It commutes with the tangent structure $\mathcal{T}$: $[X^c, \mathcal{T}]=0$.
\end{itemize}
The commutation condition $CL_2)$ is equivalent to the commutation of the two Lie derivatives $\mathcal{L}_{X^c}$ and $d_\mathcal{T}$:
\begin{eqnarray}
\mathcal{L}_{X^c}d_\mathcal{T} - d_\mathcal{T} \mathcal{L}_{X^c} = [\mathcal{L}_{X^c}, d_\mathcal{T}]=d_{[X^c, \mathcal{T}]}=0. \label{lxcdt}
\end{eqnarray}
An important  property of the complete lift with respect to any spray $S \in \mathfrak{X}(T_0M)$ (such as the canonical geodesic spray) is given by:
\begin{eqnarray} \label{tsxc}
\mathcal{T}[S, X^c]=0.
\end{eqnarray}
Equation \eqref{tsxc} implies that the horizontal component of $[S, X^c]$ vanishes, and hence $[S,X^c]$ is a vertical vector field. This property follows directly from the defining properties $CL_1)$ and $CL_2)$ of the complete lift:
\begin{eqnarray*}
0=[X^c,\mathcal{T}](S)=[X^c, \mathcal{C}]-\mathcal{T}[X^c, S]= -\mathcal{T}[X^c, S].
\end{eqnarray*} 
An immediate consequence of \eqref{tsxc} is the following commutation rule:
\begin{eqnarray}
i_{X^c}\mathcal{L}_{S}d_{\mathcal{T}}f - \mathcal{L}_{S}i_{X^c}d_{\mathcal{T}}f= i_{[X^c,S]}d_{\mathcal{T}}f = i_{\mathcal{T}[X^c, S]}df=0, \ \forall f \in C^{\infty}(TM). \label{ildt}
\end{eqnarray}
For a vector field $X\in \mathfrak{X}(M)$, its vertical lift is given by $X^v=\mathcal{T}X^c$ and it is a $(-1)$-homogeneous vector field.

\subsection{Main Characterisation Theorem: the global approach}
The following theorem establishes that 
Finslerian Killing vector fields are precisely those whose complete lifts preserve the canonical structures of a Finsler manifold: the Finsler norm, the sphere bundle, the energy, the symplectic structure, the contact structure, the Reeb vector field, the Sasaki metric, and the homogenized almost complex structure. 
\begin{theorem} \label{thmKglobal}
Consider $(M,F)$ a Finsler manifold and $X\in \mathfrak{X}(M)$. The following statements are equivalent:
\begin{itemize}
\item[i)] $X$ is a Killing vector field: $X^c(F)=0$;
\item[ii)] $X^c$ is tangent to the sphere bundle $SM$: $X^c\in \mathfrak{X}(SM)$;
\item[iii)] $f_1=X^v(E)$ is a first integral for the geodesic spray $S$: $S(X^v(E))=0$;
\item[iv)] $f_0=X^v(F)$ is a first integral for any projectively related spray $\widetilde{S}$: $\widetilde{S}(X^v(F))=0$;
\item[v)] $X^c$ is a contact vector field: $\mathcal{L}_{X^c}d_\mathcal{T}F=0$;
\item[vi)] $X^c$ preserves the Poincar\'e-Cartan $1$-form $d_\mathcal{T}E$: $\mathcal{L}_{X^c}d_\mathcal{T}E=0$;
\item[vii)] $X^c$ is a Hamiltonian vector field: $\exists f_1\in C^{\infty}(T_0M)$, $i_{X^c}\omega_E=-df_1$; 
\item[viii)] $X^c$ is a symplectic vector field: $\mathcal{L}_{X^c}dd_\mathcal{T}E=0$;
\item[ix)] $X^c$ is a Killing vector field for the Sasaki metric $G$: $\mathcal{L}_{X^c}G=0$;
\item[x)] $X^c$ is a Lie symmetry for the Reeb vector field: $[X^c, S/F]=0$;
\item[xi)] $X^c$ is almost analytic with respect to the homogenized almost complex structure $J_0$: $[X^c, J_0]=0$;
\item[xii)] $X^c$ is a Killing vector field for the homogenized Sasaki metric $G_1$: $\mathcal{L}_{X^c}G_1=0$.
\end{itemize}
\end{theorem}

\begin{proof}
The logical flowchart of the proof for the twelve equivalences is given below: 
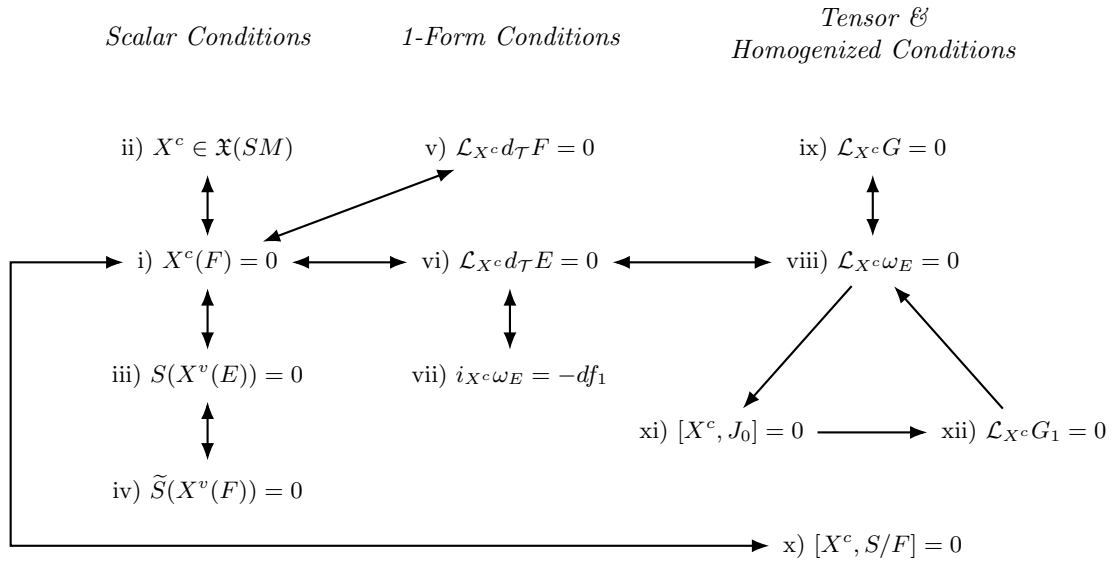
\begin{figure}[htbp]
\centering
\begin{tikzpicture}[
    >=Latex,
    % Spațiere mărită pentru a evita suprapunerea textului
    x=4.0cm, y=1.5cm,
    % Fără chenare, aliniere pe centru
    every node/.style={font=\small, align=center, draw=none, inner sep=2pt}, 
    group/.style={font=\itshape, align=center},
    eq/.style={<->, thick, shorten <=3pt, shorten >=3pt},
    imp/.style={->, thick, shorten <=3pt, shorten >=3pt}
]

% TITLURI
\node[group] at (0, 3) {Scalar Conditions};
\node[group] at (1, 3) {1-Form Conditions};
\node[group] at (2.2, 3) {Tensor \&\\Homogenized Conditions};

% COLOANA 1
\node (ii)  at (0, 2)  {$\mathrm{ii)}\ X^c\in\mathfrak{X}(SM)$};
\node (i)   at (0, 1)  {$\mathrm{i)}\ X^c(F)=0$};
\node (iii) at (0, 0)  {$\mathrm{iii)}\ S(X^{v}(E))=0$};
\node (iv)  at (0,-1)  {$\mathrm{iv)}\ \widetilde{S}(X^{v}(F))=0$};

% COLOANA 2
\node (v)   at (1, 2)  {$\mathrm{v)}\ \mathcal{L}_{X^c}d_{\mathcal{T}}F=0$};
\node (vi)  at (1, 1)  {$\mathrm{vi)}\ \mathcal{L}_{X^c}d_{\mathcal{T}}E=0$};
\node (vii) at (1, 0)  {$\mathrm{vii)}\ i_{X^c}\omega_E=-df_1$};

% COLOANA 3
\node (ix)   at (2.2, 2)   {$\mathrm{ix)}\ \mathcal{L}_{X^c}G=0$};
\node (viii) at (2.2, 1)   {$\mathrm{viii)}\ \mathcal{L}_{X^c}\omega_E=0$};
\node (xi)   at (1.7, -0.5) {$\mathrm{xi)}\ [X^c, J_0]=0$};
\node (xii)  at (2.7, -0.5) {$\mathrm{xii)}\ \mathcal{L}_{X^c}G_1=0$};
\node (x)    at (2.2, -1.5) {$\mathrm{x)}\ [X^c, S/F]=0$};

% CONEXIUNI
\draw[eq] (ii) -- (i); \draw[eq] (i) -- (iii); \draw[eq] (iii) -- (iv);
\draw[eq] (i) -- (v); \draw[eq] (i) -- (vi); \draw[eq] (vi) -- (vii);
\draw[eq] (vi) -- (viii); \draw[eq] (ix) -- (viii);

% TRIUNGHIUL LOGIC
\draw[imp] (viii) -- (xi);
\draw[imp] (xii) -- (viii);
\draw[imp] (xi) -- (xii);

% PERIFERICA
\draw[eq] (i.west) -- ++(-0.4,0) |- (x.west);

\end{tikzpicture}
%\caption{Logical flowchart of Theorem 3.2.}
\end{figure}

\emph{Equivalence of $i)$, $ii)$, $iii)$ and $iv)$ (scalar conditions):}
The first equivalence is immediate: $X$ is a Killing vector field if and only if the flow of its complete lift $Fl_t^{X^c}$ preserves the hyper-surfaces $F=\text{constant}$, meaning that $X^c$ is tangent to $SM$.

Contracting the Euler-Lagrange equations \eqref{lsdje} with the complete lift $X^c$ gives:
\begin{eqnarray*}
X^c(E) = i_{X^c}dE = i_{X^c} \mathcal{L}_Sd_\mathcal{T}E \stackrel{\eqref{ildt}}{=}  \mathcal{L}_Si_{X^c} d_\mathcal{T}E  = S(X^v(E)),  
\end{eqnarray*}
proving $i)\Longleftrightarrow iii)$.

We prove the equivalence $iii) \Longleftrightarrow iv)$. First, we recall that for the geodesic spray $S$, any projectively related spray is given by $\widetilde{S}=S-2P\mathcal{C}$, with $P\in C^{\infty}(T_0M)$ an arbitrary $1$-homogeneous function \cite[Chapter 12]{S2001}. Since $X^v$ is $(-1)$-homogeneous and $F$ is $1$-homogeneous, it follows that $f_0$ is a $0$-homogeneous function, meaning $\mathcal{C}(f_0)=0$. Thus, $\widetilde{S}(f_0)=S(f_0)$. Now, the equivalence comes from  the fact that $f_1=Ff_0$, $S(F)=0$, and therefore $S(f_1)=FS(f_0)$.

\emph{Equivalence with $v)$, $vi)$ and $vii)$ ($1$-form conditions):}
For $i) \Longrightarrow v)$ we invoke the commutation rule \eqref{lxcdt}. Thus, $X^c(F)=0$ implies $0=d_\mathcal{T}\mathcal{L}_{X^c}F=\mathcal{L}_{X^c}d_\mathcal{T}F$.

Conversely, to show $v) \Longrightarrow i)$, we use the commutation rule \eqref{ildt} for the function $F$, and $i_Sd_\mathcal{T}F=dF(\mathcal{T}S)=dF(\mathcal{C})=\mathcal{C}(F)=F$, which yields: 
\begin{eqnarray*}
0=i_S\mathcal{L}_{X^c}d_\mathcal{T}F=\mathcal{L}_{X^c}i_Sd_\mathcal{T}F=X^c(F).
\end{eqnarray*} 
 
Since conditions $i)$ and $iii)$ are equivalent, it suffices to show that either one of these two is equivalent to condition $vi)$. Given $E=F^2/2$, the direct implication follows by expanding the derivative: $\mathcal{L}_{X^c}d_\mathcal{T}E=X^c(F)d_\mathcal{T}F+F\mathcal{L}_{X^c}d_\mathcal{T}F=0$.

Assume now that condition $vi)$ is satisfied. Using again the commutation formula \eqref{ildt} for the energy function $E$, we have:
\begin{eqnarray*}
i_S\mathcal{L}_{X^c}d_\mathcal{T}E=\mathcal{L}_{X^c}i_Sd_\mathcal{T}E=2X^c(E), 
\end{eqnarray*}
since $i_Sd_\mathcal{T}E=dE(\mathcal{T}S)=dE(\mathcal{C})=\mathcal{C}(E)=2E$. Therefore, $\mathcal{L}_{X^c}d_\mathcal{T}E=0$ implies $X^c(E)=0$, which is equivalent to $X^c(F)=0$, recovering $i)$.  

For $vi \Longrightarrow vii)$, Cartan's formula allows us to write condition $vi)$ as follows: 
\begin{eqnarray} \label{ixcddje}
\mathcal{L}_{X^c}d_{\mathcal{T}}E=i_{X^c}dd_\mathcal{T}E+di_{X^c}d_\mathcal{T}E=0.
\end{eqnarray}
This shows that $X^c$ is a Hamiltonian vector field with corresponding $1$-homogeneous Hamiltonian $f_1=i_{X^c}d_\mathcal{T}E=dE(\mathcal{T}X^c)=X^v(E)$. Conversely, if $vii)$ holds for some (necessarily $1$-homogeneous) function $f_1\in C^{\infty}(T_0M)$, contracting it with the Liouville vector field we obtain:
\begin{eqnarray*}
f_1=\mathcal{C}(f_1)=i_{\mathcal{C}}df_1=-i_{\mathcal{C}}i_{X^c}dd_\mathcal{T}E= i_{X^c} i_{\mathcal{C}} dd_\mathcal{T}E = i_{X^c} \mathcal{L}_{\mathcal{C}}d_\mathcal{T}E =  i_{X^c} d_\mathcal{T}E =X^v(E).  
\end{eqnarray*}
For this function $f_1$, condition $vii)$ is equivalent to \eqref{ixcddje}, which is equivalent to the condition $vi)$.
 
\emph{Equivalence with $viii)$ and $ix)$ ($(0,2)$-type conditions):}  
Since the Lie derivative $\mathcal{L}_{X^c}$ commutes with the exterior derivative $d$, $vi) \Longrightarrow viii)$ is trivial. In order to prove the converse, we exploit the $0$-homogeneity of $X^c$, which implies the following commutation rule:
\begin{eqnarray*}
[i_{\mathcal{C}}, \mathcal{L}_{X^c}]=i_{[\mathcal{C}, X^c]}=0.
\end{eqnarray*} 
Contracting $viii)$ with the Liouville vector field $\mathcal{C}$ and using the above commutation rule, we obtain condition $vi)$:
\begin{eqnarray*}
0=i_{\mathcal{C}} \mathcal{L}_{X^c}dd_\mathcal{T}E = \mathcal{L}_{X^c} i_{\mathcal{C}} dd_\mathcal{T}E = \mathcal{L}_{X^c} \mathcal{L}_{\mathcal{C}} d_\mathcal{T}E - \mathcal{L}_{X^c} d i_{\mathcal{C}} d_\mathcal{T}E = \mathcal{L}_{X^c} d_\mathcal{T}E. 
\end{eqnarray*}
Here we used that $d_\mathcal{T}E$ is $1$-homogeneous ($\mathcal{L}_{\mathcal{C}} d_\mathcal{T}E=d_\mathcal{T}E$) and $i_{\mathcal{C}} d_\mathcal{T}E=0$.

We prove now the equivalence of $viii)$ and $ix)$. For one implication, if $X^c$ is a Killing vector field for the Sasaki metric $G$, using $\mathcal{L}_{X^c}\mathcal{T}=[X^c, \mathcal{T}]=0$ and the second formula \eqref{acs}, which gives the $2$-form $\omega_E$ in terms of the Sasaki metric $G$, we obtain $\mathcal{L}_{X^c}\omega_E=0$, which is condition $viii)$. 

For the other implication, we assume that $X^c$ preserves the symplectic form $\omega_E$. In order to prove that $X^c$ is also a Killing vector field for the Sasaki metric $G$, in view of formula \eqref{gomega}, it suffices to show $\mathcal{L}_{X^c}J=[X^c, J]=0$. 

First, we will prove that infinitesimal isometries take geodesics into geodesics, meaning $[X^c, S]=0$. Using the commutation rule of $i_S$ and $\mathcal{L}_{X^c}$ we obtain:
\begin{eqnarray*}
i_S\mathcal{L}_{X^c}\omega_E= \mathcal{L}_{X^c} i_S\omega_E + i_{[S,X^c]}\omega_E = -dX^c(E) + i_{[S,X^c]}\omega_E. \label{sxc}
\end{eqnarray*}
Using the assumption that $X$ is a Killing vector field, then $\mathcal{L}_{X^c}\omega_E=0$, $X^c(E)=0$ and since $\omega_E$ has maximal rank, we obtain that $[S,X^c]=0$.

We prove now that $[X^c, \Gamma]=0$. This follows from $[S, X^c]=0$, $[\mathcal{T}, X^c]=0$ and the Jacobi identity:
\begin{eqnarray*}
[X^c, [S, \mathcal{T}]]+ [\mathcal{T}, [X^c, S]] + [S, [\mathcal{T}, X^c]]=0.
\end{eqnarray*}
Using the definition of the horizontal projector $h=(\operatorname{Id}-[S,\mathcal{T}])/2$ we also have $[X^c, h]=0$ and in view of formula \eqref{acs} we obtain $[X^c, J]=0$. This formula and first formula \eqref{gomega} assures that $\mathcal{L}_{X^c}G=0$, hence $X^c$ is a Killing vector field for the Sasaki metric $G$.

\emph{Equivalence of $x)$ through $xii)$ (homogenized conditions)}: 
For the equivalence $i) \Longleftrightarrow x)$, expanding the Lie bracket yields:
\begin{eqnarray}
[X^c, S/F]=-\dfrac{X^c(F)}{F^2}S+\dfrac{1}{F}[X^c, S]. \label{xcsf}
\end{eqnarray}
The geodesic spray $S$ is a horizontal vector field while, according to formula \eqref{tsxc}, $[X^c, S]$ is a vertical vector field. Therefore, from \eqref{xcsf}, we have that $[X^c, S/F]=0$ if and only if $X^c(F)=0$ and $[X^c, S]=0$. While discussing condition $ix)$ above, we have proved that $X^c(F)=0$ implies $[X^c, S]=0$. This completes the proof that the condition $x)$ characterises Killing vector fields.

For a Killing vector field $X$, its complete lift preserves all the geometric structures that define the homogenized almost complex structure $J_0$ in formula \eqref{acs0}. Therefore, $\mathcal{L}_{X^c}J_0=0$, which is condition $xi)$. Assuming $[X^c, J_0]=0$ and using formula \eqref{acs0} we have that $J_0(S/F)=-\mathcal{C}$. Therefore, using also \eqref{xcsf}, we have
\begin{eqnarray*}
& 0 & =\left[X^c, J_0\right]\left(\frac{S}{F}\right)=[X^c, -\mathcal{C}]-J_0\left[X^c,\frac{S}{F}\right] \\ & = & -J_0\left( -\dfrac{X^c(F)}{F^2}S+\dfrac{1}{F}[X^c, S] \right) = -\dfrac{X^c(F)}{F} \mathcal{C} - J_0\left(\dfrac{1}{F}[X^c, S]\right),
\end{eqnarray*}
which means that $X^c(F)=0$ and $[X^c, S]=0$, hence $X$ is a Killing vector field.

We have seen that a Killing vector field satisfies $vii)$ and $xi)$ and, in view of the first formula in \eqref{g1omega}, it satisfies $xii)$. If condition $xii)$ is satisfied, using second formula \eqref{g1omega} and the fact that $\mathcal{L}_{X^c}\mathcal{T}=0$, it follows that condition $viii)$ holds, which means that $X$ is a Killing vector field. 
\end{proof}

\subsection{Comments on Killing characterisations}

In the particular case where the Finslerian norm is induced by a Riemannian metric, $F(x,y)=\sqrt{g_{ij}(x)y^iy^j}$, some of the characterisations provided in Theorem \ref{thmKglobal} had been previously obtained using alternative  techniques. Characterisation $viii)$ was established by Sasaki in \cite{S1958} using local coordinates. A global approach for characterisations $v)$, $vi)$ and $viii)$ is due to Albuquerque, \cite[Theorem 3.2]{A2019}. 

Conformal vector fields, which generalise the concept of Killing vector fields within the Finslerian setting, were studied by Szilasi and T\'oth in \cite{ST2011}. Their results can be used to recover characterisations $vi)$ and $viii)$.  

More recently, Crampin has shown in \cite{C2023} that Finslerian Killing vector fields preserve many of the canonical geometric structures, including geodesic spray, the nonlinear connection, the symplectic structure, the Sasaki metric, and the homogenized almost complex structure (hence the almost-K\"ahler structure), yielding one implication for $viii)$, $xi)$ and $xii)$. In \cite{C2023}, Crampin also proves that some other geometric structures are preserved by Killing vector fields: Cartan tensor, distortion, volume form and $S$-function. In this work we focus on those conditions that characterise Killing vector fields.

In \cite[Theorem 16]{AM2025}, Abbassi and Mekrami use some auxiliary extra conditions to obtain a characterisation for Killing vector fields of the Sasaki metric induced by a Finsler structure.

\section{Killing equations}

Conditions characterising vector fields whose flows preserve a Riemannian metric were first obtained by W. Killing in \cite{K1890}, where what we now call Killing vector fields were characterised in terms of Killing equations.

In this section, we obtain various formulations of the Killing equations in the Finslerian context, providing local expressions for the characterisations established in Theorem \ref{thmKglobal}. 

\subsection{Local expression of the geometric setting.} We now recall the local expressions and some properties of the geometric structures involved. 

For a vector field $X=X^i\frac{\partial}{\partial x^i} \in \mathfrak{X}(M)$, its complete lift is given by:
\begin{eqnarray}
X^c=X^i\frac{\partial}{\partial x^i}+ \frac{\partial X^i}{\partial x^j}y^j\frac{\partial}{\partial y^i} = X^i\frac{\partial}{\partial x^i}+ 
S(X^i)\frac{\partial}{\partial y^i} = X^i\frac{\delta}{\delta x^i}+ 
\nabla(X^i)\frac{\partial}{\partial y^i}. \label{xc}
\end{eqnarray}
For the second expression in \eqref{xc} we can use any spray $S$. In the third expression of \eqref{xc}, $\nabla X^i=S(X^i)+N^i_j X^j$ denotes the dynamical covariant derivative of the vector components $X^i$, \cite{BCD2011}. 

A fundamental property of a Finsler space is the fact that the metric tensor $g_{ij}$ and the angular metric $h_{ij}$ are covariantly constant with respect to the dynamical covariant derivative, \cite[Theorem 1.1, Theorem 1.2]{BC2024}:
\begin{eqnarray}
\nabla g_{ij}=0, \quad \nabla h_{ij}=0, \label{nablag}
\end{eqnarray}
where $\nabla g_{ij}=S(g_{ij})-N_i^kg_{kj} - N_j^kg_{ik}$ and a similar formula for $h_{ij}$.

Within the geometric setting induced by a Finsler metric, we consider the Berwald connection, which is a linear connection on $T_0M$, $\mathcal{D}: \mathfrak{X}(T_0M)\times \mathfrak{X}(T_0M) \to \mathfrak{X}(T_0M)$, preserving the horizontal and vertical distributions and has the following local expression:
\begin{eqnarray*}
\mathcal{D}_{\frac{\delta}{\delta x^i}}\frac{\delta}{\delta x^j}=F^k_{ij}\frac{\delta}{\delta x^k}, \quad \mathcal{D}_{\frac{\delta}{\delta x^i}}\frac{\partial}{\partial y^j}=F^k_{ij}\frac{\partial}{\partial y^k}, \quad 
\mathcal{D}_{\frac{\partial}{\partial y^i}}\frac{\delta}{\delta x^j}=0, \quad \mathcal{D}_{\frac{\partial}{\partial y^i}}\frac{\partial}{\partial y^j}=0. \label{Berwald_con}
\end{eqnarray*}
The local coefficients of the Berwald connection, locally defined on $T_0M$, are related to the local coefficients of the nonlinear connection as follows:
\begin{eqnarray}
F^k_{ij}=\frac{\partial N^k_i}{\partial y^j}=\frac{\partial^2 G^k}{\partial y^i\partial y^j}, \quad N^k_i=F^k_{ij}y^j. \label{Berwald_coef}
\end{eqnarray}
The second formula \eqref{Berwald_coef} is a consequence of the homogeneity. Local coefficients $N^k_i(x,y)$ are $1$-homogeneous in the fibre coordinates and by Euler's theorem we have:
\begin{eqnarray*}
N^k_i=\mathcal{C}(N^k_i)=y^j\frac{\partial N^k_i}{\partial y^j}=y^jF^k_{ij}.
\end{eqnarray*}
We denote the horizontal covariant derivative, with respect to the Berwald connection of some tensor $T$ as follows: $\mathcal{D}_{\frac{\delta}{\delta x^i}}T=T_{|i}.$ Hence, the dynamical covariant derivative is the Berwald covariant derivative in the direction of the spray $S$: 
\begin{eqnarray}
\nabla T=\mathcal{D}_ST=T_{|i}y^i. \label{dnabla}
\end{eqnarray}
For a Finsler metric $F$, the $(0,3)$-type tensor:
\begin{eqnarray*}
C_{ijk}=\frac{1}{2}\frac{\partial g_{ij}}{\partial y^k}=\frac{1}{4}\frac{\partial^3 F^2}{\partial y^i \partial y^j \partial y^k},
\end{eqnarray*}
denotes the \emph{Cartan torsion}. It is a totally symmetric tensor, vanishes if and only if the Finslerian norm is induced by a Riemannian metric, and has the property $C_{ijk}y^i=0$. 

As already mentioned, the metric tensor $g_{ij}$ is covariantly constant with respect to the dynamical covariant derivative. However, it is not covariantly constant with respect to the Berwald connection, and this can be justified as follows. We start from first formula \eqref{nablag}, which can be written using \eqref{dnabla} as $g_{ij|k}y^k=0$. If we take the derivative of this last formula with respect to the fibre coordinates, we obtain:
\begin{eqnarray}
g_{ij|k}+2\nabla C_{ijk}=0. \label{gijk}
\end{eqnarray}

\subsection{Main Characterisation Theorem: the local approach.} The next theorem provides local reformulations of the first nine characterising conditions of Theorem \ref{thmKglobal}, leading to equivalent expressions for the Killing equations. The last three characterising conditions of Theorem \ref{thmKglobal} are equivalent to $X^c(F)=0$ and their local expressions do not provide new Killing equations. 
\begin{theorem} \label{thmKlocal}
Consider $(M,F)$ a Finsler manifold and $X\in \mathfrak{X}(M)$. The following Killing equations are equivalent:

Consider $(M, F)$ a Finsler manifold and $X \in \mathfrak{X}(M)$. The following local expressions of the Killing equations are equivalent:

\begin{enumerate}
    \setlength{\itemsep}{12pt}
    
    \item[\textnormal{i)}] $X$ is a Killing vector field: 
    \[ X_{i|j}y^{i}y^{j} = 0. \]
    
    \item[\textnormal{ii)}] $X^{c}$ is tangent to the sphere bundle $SM$: 
    \[ g_{ij}y^{i}\nabla X^{j} = 0. \]
    
    \item[\textnormal{iii)}] The $1$-homogeneous function $f_{1} = g_{ij}X^{j}y^{i} = X_{i}y^{i} = X^{i}y_{i}$ is a first integral for the geodesic spray: 
    \[ \nabla(X_{i}y^{i}) = 0. \]
    
    \item[\textnormal{iv)}] The $0$-homogeneous function $f_{0} = \frac{1}{F}X_{i}y^{i} = \frac{1}{F}X^{i}y_{i}$ is a first integral for any projectively related spray $\tilde{S}$: 
    \[ \tilde{\nabla}\left(\frac{X_{i}y^{i}}{F}\right) = 0. \]
    
    \item[\textnormal{v)}] $X^{c}$ is a contact vector field: 
    \[ \nabla(h_{ij}X^{j}) + (y_{j}X^{j})_{|i} = 0. \]
    
    \item[\textnormal{vi)}] $X^{c}$ preserves the Poincar\'e-Cartan $1$-form $d_{\mathcal{T}}E$: 
    \[ (X_{i|j} + X_{j|i})y^{j} = 0. \]
    
    \item[\textnormal{vii)}] $X^{c}$ is a Hamiltonian vector field: 
    \[ X_{i} = \frac{\partial f_{1}}{\partial y^{i}}, \quad \nabla X_{i} = -\frac{\delta f_{1}}{\delta x^{i}}. \]
    
    \item[\textnormal{viii)}] $X^{c}$ is a symplectic vector field: 
    \[ X^{c}(g_{kl}) + g_{il}\frac{\partial X^{i}}{\partial x^{k}} + g_{ki}\frac{\partial X^{i}}{\partial x^{l}} = 0. \]
    
    \item[\textnormal{ix)}] $X^{c}$ is a Killing vector field for the Sasaki metric $G$: 
    \[ X_{i|j} + X_{j|i} + 2\nabla(C_{ijk}X^{k}) = 0. \]
\end{enumerate}
\end{theorem}

\begin{proof}
We prove the equivalence of the Killing equations according to the following diagram. We also show the equivalence of each Killing equation with the corresponding condition of Theorem \ref{thmKglobal}.

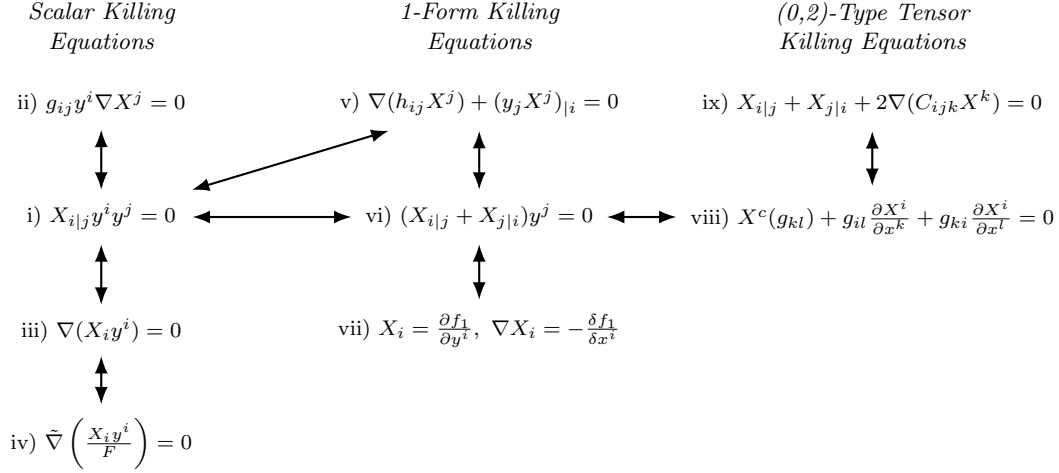
\begin{figure}[htbp]
\centering
\begin{tikzpicture}[
    >=Latex,
    every node/.style={font=\footnotesize},
    cond/.style={align=center, inner sep=5pt, minimum width=24mm}, 
    group/.style={font=\small\itshape, align=center},
    eq/.style={<->, thick},
    imp/.style={->, thick}
]

% ==================== COLOANA 1: SCALAR KILLING EQUATIONS ====================
\node[group] (T1) at (0,2.5) {Scalar Killing\\Equations};

\node[cond] (ii)  at (0,1.5)  {$\mathrm{ii)}\ g_{ij}y^i\nabla X^j = 0$};
\node[cond] (i)   at (0,0)    {$\mathrm{i)}\ X_{i|j}y^i y^j = 0$};
\node[cond] (iii) at (0,-1.5) {$\mathrm{iii)}\ \nabla(X_i y^i) = 0$};
\node[cond] (iv)  at (0,-3)   {$\mathrm{iv)}\ \tilde{\nabla}\left(\frac{X_i y^i}{F}\right) = 0$};

% Săgeți interne pe verticală clare
\draw[eq] (i) -- (ii);
\draw[eq] (i) -- (iii);
\draw[eq] (iii) -- (iv);

% ==================== COLOANA 2: 1-FORM KILLING EQUATIONS ====================
\node[group] (T2) at (5.0,2.5) {1-Form Killing\\Equations};

\node[cond] (v)   at (5.0,1.5)  {$\mathrm{v)}\ \nabla(h_{ij}X^j) + (y_j X^j)_{|i} = 0$};
\node[cond] (vi)  at (5.0,0)    {$\mathrm{vi)}\ (X_{i|j} + X_{j|i})y^j = 0$};
\node[cond] (vii) at (5.0,-1.5) {$\mathrm{vii)}\ X_i = \frac{\partial f_1}{\partial y^i},\ \nabla X_i = -\frac{\delta f_1}{\delta x^i}$};

% Săgeți interne pe verticală
\draw[eq] (vi) -- (v);
\draw[eq] (vi) -- (vii);

% ==================== COLOANA 3: (0,2)-TENSOR KILLING EQUATIONS ====================
\node[group] (T3) at (10.2,2.5) {(0,2)-Type Tensor\\Killing Equations};

\node[cond] (ix)   at (10.2,1.5)  {$\mathrm{ix)}\ X_{i|j} + X_{j|i} + 2\nabla(C_{ijk}X^k) = 0$};
\node[cond] (viii) at (10.2,0)    {$\mathrm{viii)}\ X^c(g_{kl}) + g_{il}\frac{\partial X^i}{\partial x^k} + g_{ki}\frac{\partial X^i}{\partial x^l} = 0$};

% Săgeată internă verticală
\draw[eq] (viii) -- (ix);

% ==================== LANȚUL LOGIC PRINCIPAL DINTRE COLOANE ====================

% Axa centrală a echivalențelor directe: i <=> vi <=> viii
\draw[eq] (i) -- (vi);
\draw[eq] (vi) -- (viii);

% Conexiunea oblică  i) <=> v)
\draw[eq] (i) -- (v);

% Închiderea buclei de la ix) înapoi în vi) prin exterior
%\draw[imp] (ix) -- (13.5,1.5) -- (13.5,-0.8) -- (6.8,-0.8) -- (vi);

\end{tikzpicture}
%\caption{Logical flow of the local characterisations for Killing equations grouped by their geometric types.}
\label{fig:local_proof_structure_clean}
\end{figure}

\emph{Equivalence of $i)$, $ii)$, $iii)$ and $iv)$ (scalar Killing equations)}:
For the first four Killing equations we have that $X$ is a Killing vector field if and only if $X^c(F)=0$, which is equivalent to:
\begin{eqnarray*}
\nonumber 0& = & X^c(F^2) \stackrel{\eqref{xc}}{=}\left(X^i\frac{\delta}{\delta x^i}+\nabla X^i \frac{\partial}{\partial y^i} \right) (F^2) \stackrel{\eqref{dhe}}{ = } \nabla X^i \frac{\partial F^2}{\partial y^i} \\ &  \stackrel{\eqref{fgij}}{=} & 2\nabla X^ig_{ij}y^j \stackrel{\eqref{nablag}}{=} 2y^i\nabla X_i \stackrel{\eqref{dnabla}}{=}2y^i\mathcal{D}_SX_i= 2y^i X_{i|j}y^j. \label{k12}
\end{eqnarray*}
We justify now that the second Killing equation is equivalent to the second condition of Theorem \ref{thmKglobal}. Indeed, $X^c\in \mathfrak{X}(SM)$ if and only if $G(X^c, \mathcal{C})=0$, which is equivalent to $0=g_{ij}\nabla X^i y^j=\nabla(X_iy^i)=\nabla(X_iy^i/F)$.

While, the four conditions $i)$, $ii)$, $iii)$ and $iv)$ of Theorem \ref{thmKglobal} have different geometric interpretations, their local expressions (the four scalar Killing equations) are very similar:
\begin{eqnarray*}
y^iX_{i|j}y^j=y^ig_{ij}\nabla X^j=\nabla(X_iy^i)=\nabla\left(\frac{1}{F}X_iy^i\right)=0. \label{kyy}
\end{eqnarray*}
We emphasize here that, in the Finslerian context, the components of the $(0,2)$-type tensor, $X_{i|j}$ are $0$-homogeneous and depend on the fibre coordinates, while in the Riemannian context these components do not depend on the fibre coordinates.

\emph{Equivalence with $v)$, $vi)$ and $vii)$ ($1$-form Killing equations)}: Locally, the equivalent formulation \eqref{dhe} of the Euler-Lagrange equations can be written as follows:
\begin{eqnarray*}
\frac{\delta E}{\delta x^i}=E_{|i}=0 \Longleftrightarrow \frac{\delta F}{\delta x^i}=F_{|i}=0.
\end{eqnarray*}
In these formulae, taking the derivative with respect to the fibre coordinates $y^j$ yields:
\begin{eqnarray*}
0=\frac{\partial}{\partial y^j}\left(\frac{\delta E}{\delta x^i}\right) = \frac{\partial}{\partial y^j}\left(E_{|i}\right)=\left(\frac{\partial E}{\partial y^j}\right)_{|i}, \ 0=\frac{\partial}{\partial y^j}\left(\frac{\delta F}{\delta x^i}\right) = \frac{\partial}{\partial y^j}\left(F_{|i}\right)=\left(\frac{\partial F}{\partial y^j}\right)_{|i}.
\end{eqnarray*} 
To expand the Lie derivative of the contact $1$-form $d_{\mathcal{T}}F$, we use the local expression of the complete lift $X^c$, the expression \eqref{hij} of the angular metric, and the corresponding formula above for the norm $F$. Hence: 
\begin{eqnarray*}
\nonumber \mathcal{L}_{X^c}d_{\mathcal{T}}F & = & \mathcal{L}_{X^c} \left(\frac{\partial F}{\partial y^i}dx^i \right) =X^c\left(\frac{\partial F}{\partial y^i} \right)dx^i+ \frac{\partial F}{\partial y^j} \frac{\partial X^j}{\partial x^i} dx^i \\ & = & \left(\nabla X^j \frac{\partial^2 F}{\partial y^j \partial y^i} + \frac{\partial F}{\partial y^j} X^j_{|i} \right) dx^i = \frac{1}{F} \left(h_{ij} \nabla X^j + y_j X^j_{|i}\right) dx^i, \label{k3}
\end{eqnarray*}
and hence the Killing equations $v)$ are equivalent to the corresponding condition of Theorem \ref{thmKglobal}.

Using very similar computations, if we expand the Lie derivative of the Poincar\'e-Cartan $1$-form $d_{\mathcal{T}}E$, we obtain the equivalence of the Killing equations $vi)$ and the corresponding condition of Theorem \ref{thmKglobal} :
\begin{eqnarray*}
\mathcal{L}_{X^c}d_{\mathcal{T}}E = \left( \nabla X^j \frac{\partial^2 E}{\partial y^j \partial y^i} + \frac{\partial E}{\partial y^j} X^j_{|i} \right) dx^i = \left( g_{ij} \nabla X^j + y_j X^j_{|i}\right) dx^i = \left( X_{j|i} + X_{i|j}\right) y^j dx^i. \label{k4}
\end{eqnarray*}
Alternatively, we can prove the equivalence of Killing equations $v)$ and $vi)$ using formula \eqref{hij} that connects the metric tensor and the angular metric.

We can also prove directly the equivalence of the Killing equations $i)$ and $vi)$. We emphasize that, except for the Riemannian case, the Killing equation $i)$ is not quadratic in the fibre coordinates. If we contract $vi)$ with $y^i$ we get $i)$. For the other implication, we take the derivative of $i)$ with respect to the fibre coordinates. Therefore, we have:
\begin{eqnarray*}
0=\frac{\partial}{\partial y^k}\left(X_{i|j}y^iy^j\right) = \left(X_{k|j}+X_{j|k}\right)y^j + \frac{\partial}{\partial y^k}\left(X_{i|j}\right)y^iy^j = \left(X_{k|j}+X_{j|k}\right)y^j.
\end{eqnarray*}
For the last equality above, we use $C_{ilk}y^i=0$ and $g_{il|k}y^i=-2\nabla C_{ilk}y^i=0$ to obtain: 
\begin{eqnarray}
\frac{\partial}{\partial y^k}\left(X_{i|j}\right)y^i = \frac{\partial}{\partial y^k}\left(g_{il}X^l_{|j}+g_{il|j}X^l\right)y^i = C_{ilk}y^iX^l_{|j}-2\nabla C_{ilk}y^iX^l=0. \label{partialxiji}
\end{eqnarray} 

Using the local expression \eqref{oedelta} of the symplectic $2$-form $\omega_E$, $X^c$ is a Hamiltonian vector field if and only if there exists a $1$-homogeneous function $f_1$ such that:
\begin{eqnarray}
X^c=g^{ij}\frac{\partial f_1}{\partial y^j}\frac{\delta}{\delta x^i} - g^{ij}\frac{\delta f_1}{\delta x^j}\frac{\partial}{\partial y^i}. \label{xch}
\end{eqnarray} 
If we compare \eqref{xch} with the expression of $X^c$ given by last formula \eqref{xc}, we obtain Killing equations $vii)$. 

Of the two sets of Killing equations $vii)$, the first one, $X_i=\frac{\partial f_1}{\partial y^i}$, is equivalent to $f_1=X_iy^i$. The direct implication comes from the contraction with $y^i$ and the $1$-homogeneity of $f_1$. The second implication follows by taking the derivatives with respect to the fibre coordinates. Thus, the second set of Killing equations $vii)$ can be written as follows:
\begin{eqnarray*}
0=\nabla X_i + \frac{\delta f_1}{\delta x^i} = \nabla X_i+(f_1)_{|i}= X_{i|j}y^j+ (X_jy^j)_{|i} = \left( X_{i|j}+X_{j|i}\right)y^j,
\end{eqnarray*} 
which shows the equivalence with the Killing equations $vi)$.

\emph{Equivalence with $viii)$ and $ix)$ ($(0,2)$-type form Killing equations)}: Using the local expression \eqref{oedelta} of the symplectic $2$-form $\omega_E$, we have:
\begin{eqnarray*}
\left(\mathcal{L}_{X^c} g_{ij} \delta y^i \wedge dx^j\right)\left(\frac{\partial}{\partial y^k}, \frac{\delta}{\delta x^l}\right) = X^c(g_{kl})+g_{il}\frac{\partial X^i}{\partial x^k} + g_{ki}\frac{\partial X^i}{\partial x^l},
\end{eqnarray*}
which shows the equivalence of Killing equations $viii)$ and the corresponding condition of Theorem \ref{thmKglobal}.

If we express the Lie derivative $\mathcal{L}_{X^c}G$ in terms of the covariant derivative $\mathcal{D}_{X^c}G$ with respect to the Berwald connection, we have:
\begin{eqnarray*}
\left(\mathcal{L}_{X^c}G\right)\left(\frac{\delta}{\delta x^i}, \frac{\delta}{\delta x^j}\right)& = & \left(\mathcal{D}_{X^c}G \right)\left(\frac{\delta}{\delta x^i}, \frac{\delta}{\delta x^j}\right) + G\left(\mathcal{D}_{\frac{\delta}{\delta x^i}}X^c, \frac{\delta}{\delta x^j} \right) + G\left(\frac{\delta}{\delta x^i}, \mathcal{D}_{\frac{\delta}{\delta x^j}}X^c \right) \\ &\stackrel{\eqref{xc}}{=}& 
X^kg_{ij|k} + 2\nabla X^k C_{ijk} + g_{lj}X^l_{|i} + g_{il}X^l_{|j} \stackrel{\eqref{gijk}}{=} X_{i|j} + X_{j|i} + 2\nabla\left(X^kC_{ijk}\right). 
\end{eqnarray*}
This shows the equivalence of the Killing equation $ix)$ and the condition $ix)$ of Theorem \ref{thmKglobal}.

Similarly, as we proved the equivalence of Killing equations $i)$ and $vi)$, we can prove now the equivalence of Killing equations $vi)$ and $ix)$. One implication is very simple, if we contract $ix)$ with $y^i$ and use $C_{ijk}y^i=0$, we obtain $vi)$. For the other implication, we take the derivative of $vi)$ with respect to the fibre coordinates:
\begin{eqnarray*}
\frac{\partial}{\partial y^k}\left(X_{i|j}y^j+X_{j|i}y^j\right) \stackrel{\eqref{partialxiji}}{=} X_{i|k} + X_{k|i} + \frac{\partial}{\partial y^k}\left(X_{i|j}\right) y^j = X_{i|k} + X_{k|i} + 2\nabla(C_{ilk}X^l).
\end{eqnarray*} 
For the last equality above, we used the following computations:
\begin{eqnarray*}
\frac{\partial}{\partial y^k}\left(X_{i|j}\right) y^j & = & \frac{\partial}{\partial y^k}\left(g_{il}X^l_{|j} + g_{il|j}X^l\right) y^j = 2C_{ilk}X^l_{|j}y^j - g_{il|k}X^l \\ & = & 2C_{ilk}\nabla X^l + 2X^l\nabla C_{ilk} = 2\nabla(C_{ilk}X^l).
\end{eqnarray*}

One can directly prove the equivalence of the Killing equations $viii)$ and $ix)$ by using last formula \eqref{xc}, formula \eqref{gijk}, and rewriting the partial derivatives $\partial X^i/\partial x^j$ in terms of the covariant derivatives $X^i_{|j}$.
\end{proof}

\subsection{Comments on Killing equations.} In the Riemannian context, for a vector field $X$, the components of the $(0,2)$-type tensor $X_{i|j}$ are well defined on the base manifold. Therefore, it is very easy to see the equivalence of the three types of Killing equations:  
\begin{eqnarray*}
X_{i|j}(x)y^iy^j=0 \Longleftrightarrow (X_{i|j}(x)+X_{j|i}(x))y^j=0 \Longleftrightarrow X_{i|j}(x)+X_{j|i}(x) =0.
\end{eqnarray*}
In the Finslerian case, the components of the $(0,2)$-type tensor $X_{i|j}$ are well defined on the slit tangent bundle and are $0$-homogeneous. Apart from this, the first two types of Killing equations (scalar and $1$-form) are formally the same as in the Riemannian case. 
However, the tensorial form of the Killing equations is different because of the presence of Cartan torsion, which is a purely non-Riemannian object:
\begin{eqnarray*}
X_{i|j}(x,y)y^iy^j  =  0 & \Longleftrightarrow & (X_{i|j}(x,y)+X_{j|i}(x,y))y^j=0 \\ &\Longleftrightarrow & X_{i|j}(x,y)+X_{j|i}(x,y)+2\nabla(C_{ijk}(x,y)X^k(x)) =0.
\end{eqnarray*}
In the Riemannian context, the classical Killing equations are represented by a $(0,2)$-type tensorial equation on the base manifold:
\begin{eqnarray*}
X(g_{kl}) + g_{il}\frac{\partial X^i}{\partial x^k} + g_{ki}\frac{\partial X^i}{\partial x^l} = 0.
\end{eqnarray*} 
In the Finslerian context, the components of the metric tensor $g_{ij}$ depend on the fibre coordinates, and therefore the corresponding $(0,2)$-type tensorial Killing equations are defined on the slit tangent bundle $T_0M$:  
\begin{eqnarray*}
X^c(g_{kl}) + g_{il}\frac{\partial X^i}{\partial x^k} + g_{ki}\frac{\partial X^i}{\partial x^l} = 0.
\end{eqnarray*}
In both the Riemannian and the Finslerian settings, the tensorial Killing equations are $0$-homogeneous. 

In Theorem \ref{thmKlocal}, we derive Killing equations $viii)$ from the invariance of the symplectic structure of a Finsler manifold rather than from the invariance of the Sasaki metric, since the former approach is more natural and technically simpler. In our approach, no Killing equation is regarded as privileged; each has its own geometric meaning.

%\subsection*{Acknowledgements} We express our thanks to ....

\subsection*{Disclaimer} No AI tools were used in the preparation of this manuscript, except for English language editing.


\begin{thebibliography}{99} 

\bibitem{AK2014} B. Aradi, D. Cs. Kert\'esz: \emph{Isometries, submetries and distance coordinates on Finsler
manifolds}, Acta Math. Hungar., \textbf{143} (2014), 337--350.

\bibitem{AM2025} M. Abbassi, A. Mekrami: \emph{On natural symmetries on slit tangent bundles of Finsler manifolds}, Ann. Univ. Ferrara Sez. VII Sci. Mat., \textbf{71}(2)(2025), Paper No. 28, 43 pp.

\bibitem{A2019} R. Albuquerque: \emph{Notes on the Sasaki metric}, Expo. Math., \textbf{37}(2)(2019), 207--224.

\bibitem{BC2024} I. Bucataru, O. Constantinescu: \emph{Finsler metrizabilities and geodesic invariance},
Int. J. Math., \textbf{35} (6) (2024), Article ID 2450016, 11 p.

\bibitem{BCD2011} I. Bucataru, O.A. Constantinescu, M.F. Dahl, \emph{A geometric setting for systems of ordinary differential equations},
  Int. J. Geom. Methods Mod. Phys., \textbf{8} (6) (2011), 1291--1327.

\bibitem{C2008} M. Crampin: \emph{K\"ahler and para-K\"ahler structures associated with Finsler spaces of non-zero constant flag curvature}, Houston J. Math., \textbf{34}(1)(2008), 99--114.

\bibitem{C2011} M.Crampin: \emph{Some remarks on the Finslerian version of Hilbert's fourth problem}, Houston J. Math., \textbf{37}(2) (2001),  369--391.

\bibitem{C2023} M. Crampin: \emph{On infinitesimal isometries of a Minkowskian space}, DOI: 10.13140/RG.2.2.31176.29449, 2023.

%\bibitem{C2026} M.Crampin: \emph{The Killing algebra of a Minkowskian space}, DOI: 10.13140/RG.2.2.26375.56480, 2026.

\bibitem{DH2002} S. Deng, Z. Hou: \emph{The group of isometries of a Finsler space}, Pacific J. Math., \textbf{207}(1) (2002), 149--155.

\bibitem{FN1956} A. Frölicher, A. Nijenhuis: \emph{Theory of vector-valued differential forms.} I. Derivations of the graded ring of differential forms, Indag. Math., \textbf{18} (1956), 338--359.

\bibitem{G1972} J. Grifone: \emph{Structure presque-tangente et
    connexions I}, Ann. Inst. Henri Poincare, \textbf{22} no.1 (1972),
  287--334 .
  
\bibitem{GM2000} J. Grifone, Z. Muzsnay: \emph{Variational Principles for Second Order Differential Equations. Application of the Spencer Theory to Characterize Variational Sprays},  World Scientific, 2000.   

\bibitem{K1890} W. Killing: \emph{Die Zusammensetzung der stetigen endlichen Transformationsgruppen}, Mathematische Annalen, \textbf{35} (1890), 641--676.

\bibitem{KMS1993} I. Kolář, P. Michor, J. Slovák: \emph{Natural operations in differential geometry}, Springer-Verlag, Berlin, 1993.

\bibitem{L2004} R. Lovas: \emph{Affine and projective vector fields on spray manifolds}, Periodica Mathematica Hungarica, \textbf{48} (1–2)(2004), 165--179.

\bibitem{S1958} S. Sasaki: \emph{On the differential geometry of tangent bundles of Riemannian manifolds}, Tôhoku Math. J., \textbf{10}(1958), 338--354.

\bibitem{S2001} Z. Shen: \emph{Differential geometry of spray and
Finsler spaces}, Springer, 2001.

\bibitem{ST2011} J. Szilasi, A. T\'oth: \emph{Conformal vector fields on Finsler manifolds}, Communications in Mathematics, \textbf{19} (2011) 149--168.

\bibitem{SLK2014} J. Szilasi, R. Lovas, D. Kert\'esz:
  \emph{Connections, sprays and Finsler structures}, World Scientific,
  2014.

\bibitem{Y1955} K. Yano: \emph{The Theory Of Lie Derivatives And Its Applications}, North-Holland Publishing, 1955.

\end{thebibliography}
\end{document}